\documentclass[12pt,a4paper]{amsart}
\usepackage{mathrsfs}
\usepackage{amssymb,amsmath,amsthm,color}
\usepackage{graphicx,cite}
\usepackage{hyperref}
\usepackage{url}
\usepackage{setspace}
\usepackage{enumerate}
\newtheorem{theorem}{Theorem}[section]
\newtheorem{lemma}[theorem]{Lemma}
\newtheorem{proof of lemma}[theorem]{Proof of Lemma}

\theoremstyle{definition}
\newtheorem{definition}[theorem]{Definition}

\newtheorem{remark}[theorem]{Remark}

\numberwithin{equation}{section}

\begin{document}

\title[Heisenberg uniqueness pairs]
{Non-harmonic cones are Heisenberg uniqueness pairs for the Fourier transform on $\mathbb R^n$}

\author{R. K. Srivastava}

\address{Department of Mathematics, Indian Institute of Technology, Guwahati, India 781039.}
\email{rksri@iitg.ernet.in}


\date{\today}


\begin{abstract}
In this article, we prove that a cone is a Heisenberg uniqueness pair corresponding
to sphere as long as the cone does not completely recline on the level surface of any
homogeneous harmonic polynomial on $\mathbb R^n.$ We derive that
$\left(S^2, \text{ paraboloid}\right)$ and $\left(S^2, \text{ geodesic of } S_r(o)\right)$
are Heisenberg uniqueness pairs for a class of certain symmetric finite Borel measures in
$\mathbb R^3.$ Further, we correlate the problem of Heisenberg uniqueness pairs to the
sets of injectivity for the spherical mean operator.
\end{abstract}

\maketitle

\section{Introduction}\label{section1}

A Heisenberg uniqueness pair is a pair $\left(\Gamma, \Lambda\right)$, where $\Gamma$
is a surface in $\mathbb R^n$ and $\Lambda$ is a subset of $\mathbb R^n$ such that any 
finite Borel measure $\mu$ which is supported on $\Gamma$ and absolutely continuous 
with respect to the surface measure, whose Fourier transform $\hat\mu$ vanishes on 
$\Lambda,$ implies $\mu=0.$

\smallskip

In general, the existence of Heisenberg uniqueness pair (HUP) is a question
of asking about the determining properties of the finite Borel measures which
are supported on some lower dimensional entities whose Fourier transform also
vanishes on lower dimensional entities. In fact, the main contrast in the HUP
problem to the known results on determining sets for measures \cite{BS}
is that the set $\Lambda$ has also been considered as a very thin set.
In particular, if $\Gamma$ is compact, then $\hat\mu$ is real analytic,
having exponential growth, and hence $\hat\mu$ can vanishes on a very delicate
set. Thus, the HUP problem becomes little easier in this case. However,
this problem becomes immensely difficult when the measure is supported
on a non-compact entity. It appears that the HUP problem is a natural
invariant of the theme of the uncertainty principle for the Fourier
transform.

\smallskip

In addition, the concept of determining the Heisenberg uniqueness pair for a
class of finite measures has also a significant similarity with the celebrated
result due to M. Benedicks (see \cite{B}). That is, support of a function
$f\in L^1(\mathbb R^n)$ and its Fourier transform $\hat f$ both cannot be of
finite measure simultaneously. Later, various analogues of the Benedicks theorem
has been investigated in different aspects including the Heisenberg group and
Euclidean motion groups (see \cite{NR, PS, SST}).

\smallskip

However, our main objective in this article is to discuss the concept of HUP,
which was first introduced by Hedenmalm and Montes-Rodríguez in 2011.
In the article \cite{HR}, Hedenmalm and Montes-Rodr\'iguez have shown that the pair
(hyperbola, some discrete set) is a Heisenberg uniqueness pair. As a dual problem,
a weak$^\ast$ dense subspace of $L^{\infty}(\mathbb R)$  has been constructed to
solve the Klein-Gordon equation. Further, a complete characterization of the
Heisenberg uniqueness pairs corresponding to any two parallel lines has been
given by Hedenmalm and Montes-Rodr\'iguez (see \cite{HR}). Thereafter,
a considerable amount of work has been done pertaining to the Heisenberg
uniqueness pair in the plane as well as in the Euclidean spaces.

\smallskip

Recently, Lev \cite{L} and Sj\"{o}lin \cite{S1} have independently shown that
circle and certain system of lines are HUP corresponding to the unit circle $S^1.$
Further, Vieli \cite{V1} has generalized HUP corresponding to circle in the higher
dimension and shown that a sphere whose radius does not lie in the zero set of the
Bessel functions $J_{(n+2k-2)/2};~k\in\mathbb Z_+,$ the set of non-negative integers,
is a HUP corresponding to the unit sphere $S^{n-1}.$

\smallskip

Further, Sj\"{o}lin \cite{S2} has investigated some of the Heisenberg uniqueness
pairs corresponding to the parabola. Subsequently, Babot \cite{Ba} has given a
characterization of the Heisenberg uniqueness pairs corresponding to a certain
system of three parallel lines. Thereafter, the authors in \cite{GR} have given
some necessary and sufficient conditions for the Heisenberg uniqueness pairs
corresponding to a system of four parallel lines. In the latter case, a
phenomenon of three totally disconnected interlacing sets that are given by zero
sets of three trigonometric polynomials has been observed. However, an exact analogue
for the finitely many parallel lines as compared to three lines result \cite{Ba}
is still unsolved. In \cite{GR}, the authors have also investigated some of the
Heisenberg uniqueness pairs corresponding to the spiral, hyperbola, circle and the
exponential curves.

\smallskip

In a major development, Jaming and Kellay \cite{JK} have given a unifying proof
for some of the Heisenberg uniqueness pairs corresponding to the hyperbola, polygon,
ellipse and graph of the functions $\varphi(t)=|t|^\alpha,$ whenever $\alpha>0.$
Further, Gr\"{o}chenig and Jaming \cite{GJ} have worked out some of the Heisenberg
uniqueness pairs corresponding to the quadratic surface.

\smallskip

Let $\Gamma$ be a finite disjoint union of smooth curves in $\mathbb R^2.$
Let $X(\Gamma)$ be the space of all finite complex-valued Borel measure
$\mu$ in $\mathbb R^2$ which is supported on $\Gamma$ and absolutely
continuous with respect to the arc length measure on $\Gamma$. For
$(\xi,\eta)\in\mathbb R^2,$ the Fourier transform of $\mu$ is defined by
\[\hat\mu{( \xi,\eta)}=\int_\Gamma e^{-i\pi(x\cdot\xi+ y\cdot\eta)}d\mu(x,y).\]
In the above context, the function $\hat\mu$ becomes a uniformly continuous bounded
function on $\mathbb R^2.$ Thus, we can analyze the pointwise vanishing nature of the
function $\hat\mu.$
\begin{definition}
Let $\Lambda$ be a set in $\mathbb R^2.$ The pair $\left(\Gamma, \Lambda\right)$
is called a Heisenberg uniqueness pair for $X(\Gamma)$ if any $\mu\in X(\Gamma)$
satisfies $\hat\mu\vert_\Lambda=0,$ implies $\mu=0.$

\end{definition}
Since the Fourier transform is invariant under translation and rotation, one can
easily deduce the following invariance properties about the Heisenberg uniqueness
pair.
\smallskip

\begin{enumerate}[(i)]
\item Let $u_o, v_o\in\mathbb R^2.$ Then the pair $(\Gamma,\Lambda)$
is a HUP if and only if the pair $(\Gamma+ u_o,\Lambda+v_o)$ is a HUP.

\smallskip

\item Let $T : \mathbb R^2\rightarrow \mathbb R^2$ be an invertible
linear transform whose adjoint is denoted by $T^\ast.$  Then  $(\Gamma,\Lambda)$
is a HUP if and only if $\left(T^{-1}\Gamma,T^\ast\Lambda\right)$  is a HUP.
\end{enumerate}

Now, we would like to state the first known result about the Heisenberg
uniqueness pair due to Hedenmalm and Montes-Rodríguez \cite{HR}.

\begin{theorem}{\em\cite{HR}\label{th18}}
Let $\Gamma$ be the hyperbola $x_{1}x_{2}=1$ and $\Lambda_{\alpha,\beta}$ a lattice-cross
defined by
\[\Lambda_{\alpha,\beta}=\left(\alpha\mathbb Z\times\{0\}\right)\cup\left(\{0\}\times\beta\mathbb Z\right),\]
where $\alpha, \beta$ are positive reals. Then $\left(\Gamma,\Lambda_{\alpha,\beta}\right)$ is a Heisenberg
uniqueness pair if and only if $\alpha\beta\leq1$.
\end{theorem}

For $\xi\in\Lambda,$ defining a function $e_\xi$ on $\Gamma$ by $e_{\xi}(x)=e^{i\pi x\cdot\xi}.$
As a dual problem to Theorem \ref{th18}, Hedenmalm and Montes-Rodr\'iguez \cite{HR} have proved
the following density result which in turn solve the one-dimensional Klein-Gordon equation.
\begin{theorem}{\em\cite{HR}\label{th18}}
The pair  $(\Gamma,\Lambda)$ is a Heisenberg uniqueness pair  if and only if
the set $\{e_{\xi}:~\xi\in\Lambda\}$ is a weak$^\ast$ dense subspace of
$L^{\infty}(\Gamma).$
\end{theorem}
\begin{remark}\label{rk6}
In particular, the HUP problem has another formulation. That is,  if $\Gamma$ is
the zero set of a polynomial $P$ on $\mathbb R^2,$ then $\hat\mu$ satisfies the
PDE $P\left(-i\partial\right)\hat\mu=0$ with initial condition $\hat\mu\vert_\Lambda=0.$
This may help potentially in determining the geometrical structure of the set $Z(\hat\mu),$
the zero set of the function $\hat\mu.$ If we consider $\Lambda$ to be contained in $Z(\hat\mu),$
then $\left(\Gamma, \Lambda\right)$ is not a HUP. Hence the question of the HUP arises when
$\Lambda$ has located away from $Z(\hat\mu).$
\end{remark}

\begin{definition}\label{def1}
A set $C$ in $\mathbb R^n~(n\geq 2)$ which satisfies the scaling condition
$\lambda C\subseteq C,$ for all $\lambda\in\mathbb R,$ is called a cone.
\end{definition}

Let $S^{n-1}$ denote unit sphere in $\mathbb R^n.$ In this article, we prove
that the pair $\left(S^{n-1}, C\right)$ is a Heisenberg uniqueness pair as
long as the cone $C$ does not recline on the level surface of any homogeneous harmonic
polynomial on $\mathbb R^n.$ We will call such cones as \textbf{\em non-harmonic}
cones.

\smallskip

An example of such a cone has been produced by Armitage (see \cite{A}). Let $0<\alpha<1$
and let $G_l^\lambda$ denote the Gegenbauer polynomial of degree $l$ and order $\lambda.$
Then \[K_\alpha=\left\{x\in\mathbb R^n:~|x_1|^2=\alpha^2|x|^2\right\}\] is a non-harmonic
cone if and only if $D^m G_l^{\frac{n-2}{2}}(\alpha)\neq0$ for all $0\leq m\leq l-2,$ where
$D^m$ denotes the $m$-th derivative.

\section{Notation and Preliminaries}\label{section2}
In this section, we recall certain standard facts about spherical harmonics. For more details
see \cite{T}, p. 12.

\smallskip

Let $K=SO(n)$ be the special orthogonal group and $M=SO(n-1).$ Let $\hat{K}_M$ denote
the set of all the equivalence classes of irreducible unitary representations of $K$
which have a nonzero $M$-fixed vector.  It is well known that each representation in
$\hat{K}_M$ has, in fact, a unique nonzero $M$-fixed vector, up to a scalar multiple.

\smallskip
For a $\sigma \in \hat{K}_M,$ which is realized on $V_{\sigma},$ let $\{e_1,\ldots ,e_{d(\sigma)}\}$
be an orthonormal basis of $V_{\sigma},$ with $e_1$ as the $M$-fixed vector.
Let $t^{ij}_{\sigma}(k)=\langle e_i,\sigma(k)e_j \rangle ,$ whenever $k\in K.$
By the Peter-Weyl Theorem for the representations of a compact group, it follows that
$\{\sqrt{d(\sigma)}t^{1j}_{k}:1\leq j\leq d(\sigma),\sigma\in\hat{K}_M\}$ is an
orthonormal basis of $L^2(K/M).$

\smallskip

We also need a concrete realization of the representations in $\hat{K}_M,$
which can be done in the following way.

\smallskip

Let $\mathbb Z_+$ denote the set of all non-negative integers. For $l\in \mathbb Z_+$,
let $P_l$ denote the space of all homogeneous polynomials $P$ in $n$ variables of
degree $l.$ Let $H_l=\{P\in P_l: \Delta P=0\},$ where $\Delta$ is the standard
Laplacian on $\mathbb R^n.$ The elements of $H_l$ are called solid spherical
harmonics of degree $l.$ It is easy to see that the natural action of $K$
leaves the space $H_l$ invariant. In fact, the corresponding unitary
representation $\pi_l$ is in $\hat{K}_M.$ Moreover, $\hat{K}_M$ can
be identified, up to unitary equivalence, with the collection
$\{\pi_l: l\in\mathbb Z_+.\}$

\smallskip

Define the spherical harmonics on the sphere $S^{n-1}$ by
$Y_{lj}(\omega )=\sqrt{d_l}t^{1j}_{\pi_l}(k),$ where
$\omega=k.e_n\in S^{n-1},$ $k\in K$ and $d_l$
is the dimension of $H_l.$  Then the set
$\widetilde H_l=\left\{Y_{lj}:1\leq j\leq d_l,l\in\mathbb Z_+\right\}$
forms an orthonormal basis for $L^2(S^{n-1}).$  Thus, we can expand
a suitable function $g$ on $S^{n-1}$ as
\begin{equation}\label{exp10}
g(\omega)=\sum_{l=0}^{\infty}\sum_{j=1}^{d_l}~a_{lj}Y_{lj}(\omega)
\end{equation}

\smallskip

For each fixed $\xi\in S^{n-1},$ define a linear functional on $\widetilde H_l$
by $\xi\mapsto Y_l(\xi).$ Then there exists a unique spherical harmonic,
say $Z_\xi^{(l)}\in H_l$ such that
\begin{equation}\label{exp1}
Y_l(\xi)=\int_{S^{n-1}}Z_\xi^{(l)}(\eta)Y_l(\eta)d\sigma(\eta).
\end{equation}
The spherical harmonic $Z_\xi^{(l)}$ is a $K$ bi-invariant real-valued function
which is constant on the geodesics orthogonal to the line joining the origin and
$\xi.$ The spherical harmonic $Z_\xi^{(l)}$ is called the zonal harmonic of the
space $\widetilde H_l$ around the point $\xi$ for the above and the various other
peculiar reasons. For more details, see \cite{SW}, p. 143.

\smallskip

Let $f$ be a function in $L^1(S^{n-1}).$ For each $l\in\mathbb Z_+,$ we
define the $l^{th}$ spherical harmonic projection of the function $f$ by
\begin{equation}\label{exp11}
\Pi_lf(\xi)=\int_{S^{n-1}}Z_\xi^{(l)}(\eta)f(\eta)d\sigma(\eta).
\end{equation}
Then the function $\Pi_lf$ is a spherical harmonic of degree $l.$ If for a $\delta>(n-2)/2,$ we denote
$A_l^m(\delta)=\binom{m-l+\delta}{\delta}{\binom{m+\delta}{\delta}}^{-1},$ then the spherical harmonic
expansion $\sum\limits_{l=0}^\infty\Pi_lf$ of the function $f\in L^1(\mathbb R^n)$ is $\delta$ -
Cesaro summable to $f.$ That is,
\begin{equation}\label{exp2}
f=\lim\limits_{m\rightarrow\infty}\sum_{l=0}^m~A_l^m(\delta)\Pi_lf,
\end{equation}
where limit on the right-hand side of (\ref{exp2}) exists in  $L^1\left(S^{n-1}\right).$
For more details see \cite{S}.

\smallskip
We would like to mention that the proof of our main result is carried out
by restricting the problem to the unit sphere $S^{n-1}$ in terms of averages of
its geodesic spheres. This is possible because the cone $C$ is closed under scaling.

\smallskip
For $\omega\in S^{n-1}$ and $t\in(-1, 1),$ the set $S_\omega^t=\left\{\nu\in S^{n-1}: \omega\cdot\nu=t\right\}$
is a geodesic sphere on $S^{n-1}$ with pole at $\omega.$ Let $f$ be an integrable function
on $S^{n-1}.$ Then by Fubini's Theorem, we can define the geodesic spherical means of the
function $f$ by
\[\tilde f(\omega, t)=\int_{S_\omega^t}f d\nu_{n-2},\]
where $\nu_{n-2}$ is the normalized surface measure on the geodesic sphere $S_\omega^t.$

\smallskip

Since the zonal harmonic $Z_\xi^{(l)}(\eta)$ is $K$ bi-invariant, there exists
a nice function $F$ in $(-1,1)$ such that $Z_\xi^{(l)}(\eta)=F(\xi\cdot\eta).$
Hence the extension of the formula (\ref{exp1}) becomes inevitable. An extension
of formula (\ref{exp1}) for the functions $F$ in $L^1(-1,1)$ was obtained. This
is known as the Funk-Hecke Theorem. That is,
\begin{equation}\label{exp3}
\int_{S^{n-1}}F(\xi\cdot\eta)Y_l(\eta)d\sigma(\eta)= C_lY_l(\xi),
\end{equation}
where the constant $C_l$ is given by
\[C_l=\alpha_l\int_{-1}^1F(t)G_l^{\frac{n-2}{2}}(t)(1-t^2)^{\frac{n-3}{2}}~dt\]
and $G_l^\beta$ stands for the Gegenbauer polynomial of degree $l$ and order $\beta.$
As a consequence of the Funk-Hecke Theorem, it can be deduced that the geodesic mean of a
spherical harmonic $Y_l$ can be expressed as
\begin{equation}\label{exp4}
\widetilde Y_l(\omega,t)=D_l(1-t^2)^{\frac{n-2}{2}}G_l^{\frac{n-2}{2}}(t)Y_l(\omega),
\end{equation}
where the constant $D_l=|S^{n-2}|/G_l^{\frac{n-2}{2}}(1)$ and $|S^{n-2}|$ denotes
the surface area of the unit sphere in $\mathbb R^{n-1}.$ For more details see
\cite{AAR}, p. 459. In order to prove the main result of this article, we need 
the following lemma, which percolates the geodesic mean vanishing condition of
$f\in L^1(S^{n-1})$ to each spherical harmonic component of $f$. For the class 
of continuous functions $C(S^{n-1}),$ this lemma has been worked in \cite{AVZ}. 
We prove in this article for $L^1(S^{n-1})$ using $\delta$ - Cesaro summation 
technique described above.
\begin{lemma}\label{lemma3}
Let  $f\in L^1(S^{n-1}).$ Then $\tilde f(\omega, t)=0~$ for all $t\in (-1, 1)$
if and only if $\Pi_lf(\omega)=0$ for all $l\in\mathbb Z_+.$
\end{lemma}
Notice that as a corollary to Lemma \ref{lemma3}, it can be deduced that if
$\tilde f(\omega, t)=0$ for all $t\in(-1, 1),$ then $f=0$ a.e. on $S^{n-1}$
if and only if $\omega$ is not contained in the zero set of any homogeneous
harmonic polynomial.
\begin{proof}
By the hypothesis, we have $\tilde f(\omega, t)=0~$ for all $t\in (-1, 1).$
Now, by taking geodesic mean in (\ref{exp2}) and then using (\ref{exp4}),
we arrive at
\begin{equation}\label{exp5}
\lim\limits_{m\rightarrow\infty}\sum_{k=0}^m~A_l^m(\delta)C_lG_l^{\frac{n-2}{2}}(t)\Pi_lf(\omega)=0.
\end{equation}
Since the set $\left\{G_l^{\frac{n-2}{2}}: ~l\in\mathbb Z_+\right\}$ form an orthonormal set
on $(-1,1)$ with weight $\left(1-t^2\right)^{-1/2},$ from (\ref{exp5}) it follows that
\[\lim\limits_{m\rightarrow\infty}A_l^m(\delta)C_l\left\|G_l^{\frac{n-2}{2}}\right\|_2^2\Pi_lf(\omega)=0.\]
By using the fact that for each fixed $l,$ we have $\lim\limits_{m\rightarrow\infty}A_l^m(\delta)=1,$
we conclude that $\Pi_lf(\omega)=0$ for all $l\in\mathbb Z_+.$  In particular, if $\omega$
is not contained in $Y_l^{-1}(0)$ for all $l\in\mathbb Z_+,$ then $f(\omega)=0$ a.e.
on $S^{n-1}.$  This completes the proof of Lemma \ref{lemma3}.
\end{proof}

\section{Proofs of the main result}\label{section3}
In this section, we first prove that a non-harmonic cone is a Heisenberg 
uniqueness pair corresponding to the unit sphere.

\begin{theorem}\label{th11}
Let $\Lambda=C$ be a cone in $\mathbb R^n.$ Then $\left(S^{n-1},\Lambda\right)$
is a Heisenberg uniqueness pair if and only if $\Lambda$ is not contained in
$P^{-1}(0),$  whenever $P\in H_l$ and $l\in\mathbb Z_+.$
\end{theorem}
\begin{proof}
Since $\mu$ is absolutely continuous with respect to the surface measure on $S^{n-1},$
by Radon-Nikodym theorem, there exists a function $f$ in $L^{1}\left(S^{n-1}\right)$
such that $d\mu=f(\eta)d\sigma(\eta),$ where $d\sigma$ is the normalized surface measure
on $S^{n-1}.$ Suppose $\hat\mu\vert_\Lambda= 0.$ Then
\begin{equation}\label{exp6}
\hat\mu(\xi)=\int_{S^{n-1}} e^{-i{\xi\cdot\eta}}f(\eta)d\sigma(\eta)=0
\end{equation}
for all $\xi\in S^{n-1}.$ Let $\xi=r\omega,$ where $r>0$ and $\omega\in S^{n-1}.$
By decomposing the integral in (\ref{exp6}) into the
geodesic spheres at pole $\omega,$ we get
\[\int_{-1}^1\left(\int_{S_\omega^t} e^{-i{r \omega\cdot\nu}}f(\nu)d\sigma_{n-2}(\nu)\right)dt=0,\]
where $S_\omega^t=\left\{\nu\in S^{d-1}: \omega\cdot\nu=t\right\}.$  That is,
\begin{equation}\label{exp7}
\int_{-1}^1e^{i r t}\tilde f(\omega, t)dt=0,
\end{equation}
for all $r>0.$ Since $f\in L^{1}\left(S^{n-1}\right),$ the geodesic mean
$\tilde f(\omega, t)$ will be a continuous function on $(-1,1).$ Thus for each fixed $\omega,$
the left-hand side of (\ref{exp7}) can be viewed as the Fourier transform of the
compactly supported function $\tilde f(\omega, .)$ on $\mathbb R.$ Hence, it can be
extended holomorphically to $\mathbb C.$ Then, in this case, the Fourier transform of
$\tilde f(\omega, .)$ can vanish at most on a countable set. Thus, by the continuity
of $\tilde f(\omega, .)$ it follows that $\tilde f(\omega, t)=0$ for all $t\in (-1, 1).$
Hence, in view of Lemma \ref{lemma3}, we conclude that $f=0$ a.e. on $S^{n-1}$ if and
only if $\omega$ is not contained in $Y_l^{-1}(0)$ for all $l\in\mathbb Z_+.$ Since the
cone $\Lambda$ is closed under scaling, we infer that $f=0$ a.e. if and only if $\Lambda$
is not contained in $P^{-1}(0)$ for any $P\in H_l$ and for all $l\in\mathbb Z_+.$
Thus $\mu=0.$
\smallskip

Conversely, suppose the cone $C$ is contained in the zero set of a homogeneous harmonic
polynomial, say $P_j\in H_l.$ Then, we can construct a finite complex Borel measure $\mu$
in $\mathbb R^n$ such that  $d\mu=Y_j(\eta)d\sigma(\eta),$ where $Y_j\in\widetilde H_l.$

\smallskip

Using the Funk-Hecke Theorem, it has been shown that for spherical harmonic
$Y_j\in\widetilde H_l,$ the following identity holds.
\begin{equation}\label{exp16}
\int_{S^{n-1}} e^{-i{x\cdot\eta}}Y_j(\eta)d\sigma(\eta)=i^j~\frac{J_{j+(n-2)/2}(r)}{r^{(n-2)/2}}Y_j(\xi),
\end{equation}
where $x=r\xi,$ for some $r>0.$ For a proof of identity (\ref{exp16}), see \cite{AAR}, p. 464.
This in turn implies that $\hat\mu\vert_C= 0.$

\end{proof}

\begin{remark}\label{rk1}
$\textbf{(a).}$
A set which is determining set for any real analytic function is called $NA$ - set.
For instance, the spiral is an $NA$ - set in the plane (see \cite{PS_0}). The set
\[\Lambda_\varphi=\left\{(x_1,x_2,x_3): x_3(x_1^2+x_2^2)=x_1\varphi(x_3)\right\},\]
where function $\varphi$ is given by $\varphi(x_3)=\exp{\dfrac{1}{x_3^2-1}}\,,$
for $|x_3|<1$ and $0$ otherwise. The set $\Lambda_\varphi$ is an $NA$ - set. For more details
see \cite{PS_0}. Since the Fourier transform of a finite Borel measure $\mu$ which is supported
on the boundary $\partial\Omega$ of a bounded domain $\Omega$ in $\mathbb R^n$ can
be extended holomorphically to $\mathbb C^n,$ the pair $(\partial\Omega,\text{NA - set})$
is a Heisenberg uniqueness pair. However, the converse is not true. Hence, all together with
the result of Vieli \cite{V1}, it is an interesting question to examine,
whether the exceptional sets for the HUPs corresponding to $\Gamma=S^{n-1},$ are
eventually contained in the zero sets of all homogeneous harmonic polynomials and the
countably many spheres whose radii are contained in the zero set of the certain class
of Bessel functions. We leave it open for the time being.

\smallskip

\noindent $\textbf{(b).}$ For $\Gamma=S^{n-1},$ it is easy to verify that $\hat\mu$
satisfies Helmholtz's equation
\begin{equation}\label{exp8}
\Delta\hat\mu + \hat\mu=0
\end{equation}
with initial condition $\hat\mu\vert_\Lambda=0.$ For a continuous function $f$ on
$\mathbb R^n$ ($n\geq2),$ the spherical mean $Rf$ of $f$ over the sphere
$S_r(x)=\{y\in\mathbb R^n:~ |x-y|=r\}$ is defined by
\[Rf(x,r)=\int_{S_r(x)}f(y)d\sigma_r(y),\]
where $d\sigma_r$ is the normalized surface measure on the sphere $S_r(x).$
Then $\hat\mu$ will satisfy the functional equation
\begin{equation}\label{exp9}
R\hat\mu(x,r)=c_n\frac{J_{(n-2)/2}(r)}{r^{(n-2)/2}}\hat\mu(x).
\end{equation}
Thus, we infer that  $\hat\mu(x)=0$ if and only if $R\hat\mu(x,r)=0$ for all $r>0.$

\smallskip

In an interesting article by Zalcman et al. \cite{AVZ}, it is shown that
for $f$ to be continuous function on $\mathbb R^n$ if $Rf(x,r)=0$ for all $r>0$
and for all $x\in C,$ then $f\equiv0$ if and only if $C$ is a non-harmonic cone
in $\mathbb R^n.$ In integral geometry, such sets are called sets of injectivity
for the spherical means. We do not digress here to give more history of sets of
injectivity for the spherical means in various set ups, still, we would like to
refer to \cite{ABK,AK,AN2,AQ,AR,AVZ,NRR,NT1, Sri, Sri2, Sri3}. However, this is
an incomplete list of the articles on the sets of injectivity.

\smallskip

Thus, in view of the above result, it follows that $\hat\mu\equiv0$
if and only if $C$ is a non-harmonic cone in $\mathbb R^n.$ As $\mu$ is
a signed measure, we again need to go through the proof of  Theorem \ref{th11},
in order to show that $\mu=0.$

\smallskip
Now, consider $\Lambda$ to be an arbitrary set in $\mathbb R^n.$  Then, it is clear that
$(S^{n-1}, \Lambda)$  is HUP if and only if $\Lambda$ is a set of injectivity for
spherical mean over a class of certain real analytic functions. However, the
latter problem is yet not settled.
\end{remark}

\section{Some observations for a special class of measures in $\mathbb R^3$}\label{section4}

In this section, we shall prove that the paraboloid is a HUP corresponding to
the unit sphere $S^2$ in $\mathbb R^3$ for a class of finite Borel measure which
are given by certain symmetric functions in $L^1(S^2).$ Further, we prove that
a geodesic on the sphere $S_R(o)$ is a HUP corresponding to $S^2$ for the
above class of measures.

\smallskip

We need the following lemma for proofs of our results of this section.

\begin{lemma}\label{lemma1}
Let $f\in L^1\left(S^{n-1}\right)$ be such that
$\int\limits_{S^{n-1}} e^{-i{x\cdot\eta}}f(\eta)d\sigma(\eta)=0.$
Then
\begin{equation}\label{exp12}
\lim\limits_{m\rightarrow\infty}\sum_{k=0}^mi^kA_k^m~\frac{J_{k+(n-2)/2}( r)}{r^{(n-2)/2}}~\Pi_kf(\xi)=0,
\end{equation}
where $x=r\xi,$ for some $r>0$ and $\xi\in S^{n-1}.$
\end{lemma}
\begin{proof}\
We have
\begin{eqnarray*}
&&  \left|\sum_{k=0}^mA_k^m\int_{S^{n-1}} e^{- i{x\cdot\eta}}\Pi_kf(\eta)d\sigma(\eta)\right| \\
&=& \left|\sum_{k=0}^m\int_{S^{n-1}}e^{- i{x\cdot\eta}}(A_k^m\Pi_kf\left(\eta)-f(\eta)\right)d\sigma(\eta)\right|\\
&\leq & \sum_{k=0}^m\int_{S^{n-1}}\left|(A_k^m\Pi_kf\left(\eta)-f(\eta)\right)\right|d\sigma(\eta).
\end{eqnarray*}
In view of (\ref{exp2}), it follows that
\begin{equation}\label{exp18}
\lim\limits_{m\rightarrow\infty}\sum_{k=0}^mA_k^m\int_{S^{n-1}} e^{- i{x\cdot\eta}}\Pi_kf(\eta)d\sigma(\eta)=0.
\end{equation}
This in turn from (\ref{exp16}) implies that (\ref{exp12}) holds.
\end{proof}

We know that for $n=3,$ a typical spherical harmonic of degree $k$ can be
expressed as $Y_k^l(\theta,\varphi)=e^{il\varphi}P_k^l(\cos\theta),$ where
$P_k^l$'s are the associated Legendre functions. In fact, the set
$\{Y_k^l:\, -k\leq l\leq k\}$ forms an orthonormal basis for $\widetilde H_k,$
(see \cite{Su}, p. 91). Hence, the $k$-th spherical harmonic projection $\Pi_kf$
can be expressed as
\[\Pi_kf(\theta,\varphi)=\sum_{l=-k}^kC_k^l(f)e^{il\varphi}P_k^l(\cos\theta),\]
where $0\leq\theta<\pi$ and $0\leq\varphi<2\pi.$ Thus, an integrable function $f$
on $S^2$ has the spherical harmonic expansion as
\begin{equation}\label{exp13}
f(\theta,\varphi)=\sum_{k=0}^\infty\sum_{l=-k}^kC_k^l(f)e^{il\varphi}P_k^l(\cos\theta).
\end{equation}
Let $L^1_{\text{sym}}(S^2)$ denotes the space of all those functions $f$ in $L^1(S^2)$
that satisfy a set of {\em symmetric-coefficient conditions} $C_k^l(f)=C_{k'}^l(f),$
for $|l|\leq\min\{k, k'\}.$

\begin{theorem}\label{th9}
Let $\Lambda=\{(x_1,x_2,x_3):~ x_3=x_1^2+x_2^2\}.$ Then $\left(S^2,\Lambda\right)$ is
a Heisenberg uniqueness pair with respect to $L^1_{\text{sym}}(S^2).$
\end{theorem}
\begin{proof}
Since $\mu$ is absolutely continuous with respect to the surface measure on $S^2,$
there exists a function $f\in L_{\text{sym}}^{1}\left(S^{2}\right)$ such that
$d\mu=f(\eta)d\sigma(\eta),$ where $d\sigma$ is the normalized surface measure on $S^{2}.$
Suppose $\widehat\mu\vert_\Lambda= 0.$ Then
\begin{equation}\label{exp19}
\int_{S^{2}} e^{- i{\xi\cdot\eta}}f(\eta)d\sigma(\eta)=0
\end{equation}
for all $\xi\in S^2.$ Now, consider the spherical polar co-ordinates $x_1=r\sin\theta\cos\varphi,$
$~x_2=r\sin\theta\sin\varphi$ and $x_3=r\cos\theta,$ where $0\leq\theta<\pi$ and $0\leq\varphi<2\pi.$
Then, in view of Lemma \ref{lemma1}, Equation (\ref{exp19}) becomes
\begin{equation}\label{exp14}
\lim\limits_{m\rightarrow\infty}\sum_{k=0}^mi^kA_k^m\,J_{\frac{k+1}{2}}( r)~\Pi_kf(\theta, \varphi)=0
\end{equation}
for all $\varphi\in[0, 2\pi).$
Notice that the rotation $\varphi$ is independent of the choice of $r,$ because,
the paraboloid is completely determined by $\cos\theta=r\sin^2\theta.$ Since the set
$\{e^{il\varphi}:~l\in\mathbb Z_+\}$ form an orthonormal set in $L^2[0, 2\pi)$ and
$f\in L^1_{\text{sym}}(S^2),$ a simple calculation gives
\begin{equation}\label{exp15}
\int_0^{}\Pi_kf(\theta, \varphi)\overline{\Pi_df(\theta, \varphi)}d\varphi=\left\{
\begin{array}{ll}
\left\|\Pi_kf(\theta, .)\right\|_2^2  , & \text{~~~~~if~}  k<d \\
                           & \\
\left\|\Pi_df(\theta, .)\right\|_2^2  , & \text{~~~~~if~}  k\geq d.
\end{array}
\right.
\end{equation}
After multiplying (\ref{exp14}) by $\overline{\Pi_df(\theta, \varphi)}$ and
using (\ref{exp15}), we conclude that
\[\lim\limits_{m\rightarrow\infty}\left[\sum_{k=0}^{d-1}A_k^m\left|J_{\frac{k+1}{2}}(r)\right|^2\left\|\Pi_kf(\theta, .)\right\|_2^2+
\sum_{k=d}^{m}A_k^m\left|J_{\frac{k+1}{2}}(r)\right|^2\left\|\Pi_df(\theta, .)\right\|_2^2\right]=0.\]
Thus, using the fact that $\lim\limits_{m\rightarrow\infty}A_k^m=1$ and the second sum
goes to zero as $d\rightarrow\infty,$ we obtain that
\[\sum_{l=0}^{\infty}\left|J_{\frac{l+1}{2}}( r)\right|^2\left\|\Pi_lf(\theta, .)\right\|_2^2=0.\]
That is,
$|J_{\frac{l+1}{2}}( r)|\left\|\Pi_lf(\theta, .)\right\|_2=0$
for all $r>0.$ Since the Bessel functions can have at most countably many zeros,
it follows that
\[\Pi_lf(\theta,\varphi)=\sum_{d=-l}^lC_d^l(f)e^{id\varphi}P_l^d(\cos\theta)=0.\]
This in turn, because of orthogonality of the set $\{e^{il\varphi}:~l\in\mathbb Z_+\},$
implies that $C_d^l(f)P_l^d(\cos\theta)=0.$  However, on the paraboloid, we have
$\cos\theta=r\sin^2\theta,$ which gives $\cos\theta=\frac{-1+\sqrt{1+4r^2}}{2r}.$
Since the Legendre functions can vanish only at countably many points, it follows
that $C_d^l(f)=0$ for all $d$ with $-l\leq d\leq l.$ That is, $\Pi_lf=0$
for all $l\in\mathbb Z_+.$ Thus $f=0$ a.e. This completes the proof.
\end{proof}

\begin{remark}\label{rk2}
\noindent We observe that Theorem \ref{th9} could be extended to
higher dimensions in a similar way. However, to avoid the complexities of
notation and calculation, we prove the result for $n=3.$
\end{remark}

Next, we prove that a geodesic sphere which is parallel to the equator of
the sphere $S_R(o)$ is a HUP corresponding to the unit sphere $S^2$
with respect to $L^1_{\text{sym}}(S^2).$

\begin{theorem}\label{th8}
Let $\Lambda_{\alpha, R}=\{(\alpha, \varphi):~R\cos\alpha=r \text{ and }0\leq\varphi<2\pi\}.$
Then $\left(S^2,\Lambda_{\alpha, R}\right)$ is a HUP if and only if $J_{\frac{l+1}{2}}(R)\neq0$
for all $l\in\mathbb Z_+$ and the ratio $r/R$ is not contained in the zero set of any Legendre
function.
\end{theorem}
\begin{proof}
Suppose $\widehat\mu\vert_{\Lambda_{\alpha, R}}= 0.$  Then similarly
the proof of Theorem \ref{th9}, we reach the conclusion that
$|J_{\frac{l+1}{2}}( R)|\left\|\Pi_lf(\alpha, .)\right\|_2=0.$
Then $\left\|\Pi_lf(\alpha, .)\right\|_2=0$ for all $l\in\mathbb Z_+$
if $|J_{\frac{l+1}{2}}( R)|\neq0$ for all $l\in\mathbb Z_+.$ That is,
\[\Pi_lf(\alpha,\varphi)=\sum_{d=-l}^lC_d^l(f)e^{id\varphi}P_l^d(\cos\alpha)=0.\]
By the uniqueness of the Fourier series, it follows that
$C_d^l(f)P_l^d\left(\frac{r}{R}\right)=0.$ Then $C_d^l(f)=0$ if
$P_l^d\left(\frac{r}{R}\right)\neq0.$ Under the assumptions of the
hypothesis, we conclude that $\Pi_lf=0$ for all $l\in\mathbb Z_+.$
Thus $f=0.$

\smallskip

Conversely, if either of the conditions of Theorem \ref{th8} fails, then for the measure
$d\mu=e^{il\varphi}P_k^l(\cos\theta)d\sigma(\theta, \varphi),$ it follows from the Funk-Hecke
identity (\ref{exp16}) that $\widehat\mu\vert_{\Lambda_{\alpha, R}}= 0.$ This complete
the proof.
\end{proof}

\begin{remark}\label{rk3}
It is reasonable to mention that if Theorem \ref{th8} can be extended to a general
class of finite Borel measures, then this result would have a sharp contrast, in terms
of the topological dimension of the pairing set, with the known results for HUP corresponding
to sphere.
\end{remark}

\smallskip

\noindent{\bf Concluding remarks:}

\smallskip

In this article, we have shown that $(S^{n-1}, C)$ is a HUP as long as the cone $C$
is not contained in the zero set of any homogeneous harmonic polynomial. Now, it is
natural to consider a compact subgroup $K$ of $SO(n)$ with $K_o$ the orbit of $K$
around the origin. Let $\Gamma_K=K/K_o.$ We know that a unitary irreducible representation
of $SO(n)$ can be decomposed into finitely many irreducible representations of $K.$ Thus,
the action of the group $K$ on a spherical harmonic $Y_l$ on the unit sphere $S^{n-1}$ will
decompose $Y_l$ uniquely into a finite sum of spherical harmonics. Therefore, it would be an
interesting question to find out the possibility that $(\Gamma_K, C)$ is a HUP as long as the
cone $C$ does not recline on the level surface of any $K$ - invariant homogeneous polynomial.
We leave this question open for the time being.

\smallskip

\noindent{\bf Acknowledgements:}

\smallskip

The author wishes to thank E. K. Narayanan and Rama Rawat for several fruitful discussions.
The author would also like to gratefully acknowledge the support provided by IIT Guwahati,
Government of India.

\bigskip



\begin{thebibliography}{1000}
\bibitem{ABK} M. L. Agranovsky, C. Berenstein and P. Kuchment, {\em Approximation by spherical waves in
$L^p$-spaces,} J. Geom. Anal. 6 (1996), no. 3, 365--383 (1997).

\bibitem{AN2} M. L. Agranovsky and E.K.  Narayanan, {\em Injectivity of the spherical mean operator on the
conical manifolds of spheres,} Siberian Math. J. 45 (2004), no. 4, 597 - 605.

\bibitem{AQ} M. L. Agranovsky and E. T. Quinto, {\em Injectivity sets for the Radon transform over circles
and complete systems of radial functions,} J. Funct. Anal.,  139  (1996),  no. 2, 383--414.

\bibitem{AR} M. L. Agranovsky and R. Rawat, {\em Injectivity sets for spherical means on the Heisenberg group,}
J. Fourier Anal. Appl., 5 (1999), no. 4, 363--372.

\bibitem{AVZ} M. L. Agranovsky, V. V. Volchkov and L. A. Zalcman, {\em Conical uniqueness sets for the
spherical Radon transform,} Bull. London Math. Soc. 31 (1999), no. 2, 231-–236.

\bibitem{AK}  G. Ambartsoumian and P. Kuchment, {\em On the injectivity of the circular Radon transform,}
 Inverse Problems 21 (2005), no. 2, 473--485.
 
\bibitem{AAR} G. E. Andrews, R. Askey and R. Roy, {\em Special Functions,} Cambridge University Press, Cambridge, 1999.

\bibitem{A} D. H. Armitage, {\em Cones on which entire harmonic functions can vanish,}
Proc. Roy. Irish Acad. Sect. A 92 (1992), no. 1, 107–110.

\bibitem{Ba} D. B. Babot, {\em Heisenberg uniqueness pairs in the plane, Three parallel lines,}
Proc. Amer. Math. Soc. 141 (2013), no. 11, 3899-3904.

\bibitem{BS} S. C. Bagchi and A. Sitaram, {\em  Determining sets for measures on $\mathbb R^n,$}
Illinois J. Math. 26 (1982), no. 3, 419-422.

\bibitem{B} M. Benedicks, {\em On Fourier transforms of functions supported on sets of finite Lebesgue measure,}
J. Math. Anal. Appl. 106 (1985), no.1, 180-183.

\bibitem{CH} R. Courant and D. Hilbert, {\em Methods of mathematical physics,} Vol. 2, Interscience, New York, 1962.

\bibitem{GR} D. K. Giri and R. K. Srivastava, {\em Heisenberg uniqueness pairs for some algebraic curves in the plane,}
Adv. Math., 310 (2017), 993-1016.

\bibitem{GJ} K. Gr\"{o}chenig and  P. Jaming, {\em The Cram\'{e}r-Wold theorem on quadratic surfaces and Heisenberg uniqueness pairs,}
arXiv:1608.06738.

\bibitem{HR} H. Hedenmalm, A. Montes-Rodr\'iguez, {\em Heisenberg uniqueness pairs and the Klein-Gordon equation,}
 Ann. of Math. (2) 173 (2011), no. 3, 1507-1527.

\bibitem{JK} P. Jaming and K. Kellay, {\em A dynamical system approach to Heisenberg uniqueness pairs,}
arXiv:1312.6236, July, 2014.

\bibitem{L} N. Lev, {\em Uniqueness theorem for Fourier transform. Bull. Sc. Math.,}
135(2011), 134-140.

\bibitem{NR} E. K. Narayanan and P. K. Ratnakumar, {\em Benedick's theorem for the Heisenberg group,}
Proc. Amer. Math. Soc. 138 (2010), no. 6, 2135-2140.

\bibitem{NRR} E. K. Narayanan, R. Rawat and S. K.  Ray, {\em Approximation by $K$-finite functions in
$L\sp p$ spaces,} Israel J. Math.  161  (2007), 187-207.

\bibitem{NT1} E. K. Narayanan and S. Thangavelu, {\em Injectivity sets for spherical means on the
Heisenberg group,} J. Math. Anal. Appl. 263 (2001), no. 2, 565-579.

\bibitem{PS_0} V. Pati and A. Sitaram, {\em Some questions on integral geometry on Riemannian manifolds,}
 Ergodic theory and harmonic analysis (Mumbai, 1999). Sankhya Ser. A 62 (2000), no. 3, 419-424.

\bibitem{PS} F. J. Price and A. Sitaram, {\em Functions and their Fourier transforms with supports
of finite measure for certain locally compact groups,} J. Funct. Anal. 79 (1988), no. 1, 166–182.

\bibitem{SST} A. Sitaram, M. Sundari and S. Thangavelu, {\em Uncertainty principles on certain Lie groups,}
Proc. Indian Acad. Sci. Math. Sci. 105 (1995), no. 2, 135-151.

\bibitem{S1} P. Sj\"olin, {\em Heisenberg uniqueness pairs and a theorem of Beurling and Malliavin,} Bull. Sc. Math.,
135(2011), 125-133.

\bibitem{S2} P. Sj\"olin, {\em Heisenberg uniqueness pairs for the parabola. Jour. Four. Anal. Appl.,}
19(2013), 410-416.

\bibitem{S} C. D. Sogge, {\em Oscillatory integrals and spherical harmonics,} Duke Math. J. 53 (1986), 43-65.

\bibitem{Sri} R. K. Srivastava, {\em Sets of injectivity for weighted twisted spherical means and support theorems,}
J. Fourier Anal. Appl.,  18 (2012), no. 3, 592--608.

\bibitem{Sri2} R. K. Srivastava, {\em Coxeter system of lines and planes are sets of injectivity for the twisted
spherical means,} J. Funct. Anal. 267 (2014) 352–383.

\bibitem{Sri3} R. K. Srivastava, {\em Real analytic expansion of spectral projection and extension of Hecke-Bochner identity,}
 Israel J. Math. 200 (2014), 1-22.
 
\bibitem{SW} E. M. Stein and G. Weiss, {\em Introduction to Fourier analysis on Euclidean spaces,}
 Princeton Mathematical Series, No. 32. Princeton University Press (1971).

\bibitem{Su} M. Sugiura, {\em Unitary representations and harmonic analysis}, North-Holland Mathematical Library, 44.
North-Holland Publishing Co., Amsterdam; Kodansha, Ltd., Tokyo, 1990.

\bibitem{T} S. Thangavelu, {\em An introduction to the uncertainty principle}, Prog. Math. 217, Birkhauser, Boston (2004).

\bibitem{V1} F. J. G. Vieli,  {\em A uniqueness result for the Fourier transform of measures on the sphere,}
Bull. Aust. Math. Soc. 86(2012), 78-82.

\bibitem{V2} F. J. G. Vieli,  {\em A uniqueness result for the Fourier transform of measures on the paraboloid,}
 Matematicki Vesnik, Vol. 67(2015), No. 1, 52-55.

\end{thebibliography}
\end{document}